\documentclass[a4paper,12pt]{article}

\usepackage{color,tikz,amsmath,amssymb,bm}

\unitlength\textwidth
\divide\unitlength by 200\relax

\bmdefine{\sss}{s}
\bmdefine{\vvv}{v}

\DeclareMathAlphabet{\mathscr}{U}{rsfs}{m}{n}

\newcommand{\msPPP}{\mathscr{P}}

\newcommand{\msXXX}{\mathscr{X}}
\newcommand{\msKKK}{\mathscr{K}}

\newcommand{\NNN}{\mathbb{N}}
\newcommand{\ZZZ}{\mathbb{Z}}
\newcommand{\QQQ}{\mathbb{Q}}
\newcommand{\RRR}{\mathbb{R}}
\newcommand{\KKK}{\mathbb{K}}

\newcommand{\mmmm}{\mathfrak{m}}
\newcommand{\pppp}{\mathfrak{p}}

\newcommand{\UUUUU}{{\mathcal U}}

\newcommand{\tUUUUU}{{{t}\mathcal U}}
\newcommand{\qUUUUU}{{{q}\mathcal U}}

\newcommand{\UUUUUn}{{\mathcal U}^{(n)}}

\newcommand{\define}{\mathrel{:=}}

\newcommand{\gor}{Gorenstein}
\newcommand{\cm}{Cohen-Macaulay}

\newcommand{\relint}{{\rm{relint}}}

\newcommand{\trace}{{\mathrm{tr}}}

\newcommand{\Div}{{\mathrm{Div}}}
\renewcommand{\hom}{{\mathrm{Hom}}}
\newcommand{\spec}{{\mathrm{Spec}}}

\newcommand{\aff}{\mathrm{aff}}
\newcommand{\conv}{\mathrm{conv}}
\newcommand{\stab}{\mathrm{STAB}}
\newcommand{\hstab}{\mathrm{HSTAB}}
\newcommand{\tstab}{\mathrm{TSTAB}}
\newcommand{\qstab}{\mathrm{QSTAB}}

\newcommand{\xip}{\xi^{+}}

\newcommand{\ekp}{E_\KKK[\msPPP]}

\newcommand{\eksg}{{E_\KKK[\stab(G)]}}
\newcommand{\ekhsg}{{E_\KKK[\hstab(G)]}}
\newcommand{\ektsg}{{E_\KKK[\tstab(G)]}}
\newcommand{\ekqsg}{{E_\KKK[\qstab(G)]}}

\newtheorem{thm}{Theorem}[section]
\newtheorem{fact}[thm]{Fact}

\newtheorem{lemma}[thm]{Lemma}
\newtheorem{cor}[thm]{Corollary}
\newtheorem{definition}[thm]{Definition}
\newtheorem{prop}[thm]{Proposition}
\newtheorem{remark}[thm]{Remark}

\newcommand{\bigzerou}{\smash{\lower1.7ex\hbox{\bg 0}}}

\newcommand{\bigastu}{\smash{\lower1.7ex\hbox{\bg *}}}

\numberwithin{equation}{section}

\newcommand{\mylabel}[1]{{\label{#1}\tt [#1]}}
\let\mylabel=\label

\title{%
\gor\ on the punctured spectrum and nearly \gor\ property of
the Ehrhart ring of the stable set polytope
of an h-perfect graph%
}
\author{Mitsuhiro Miyazaki%
\footnote{Partially supported by JSPS KAKENHI 20K03556.}
\\
Kyoto University of Education\\
\tt e-mail: mmyzk7@gmail.com}
\date{}

\begin{document}

\maketitle

\markright{Gorenstein Ehrhart rings}

\sloppy

\begin{abstract}
In this paper,
we give a criterion of the nearly \gor\  property of
the Ehrhart ring of the stable set polytope of
an h-perfect graph:
the Ehrhart ring of the stable set polytope of an h-perfect graph $G$ 
with connected components $G^{(1)}$, \ldots, $G^{(\ell)}$
is
nearly \gor\ if and only if
(1)
for each $i$, the Ehrhart ring of the stable set polytope of $G^{(i)}$ is \gor\
and 
(2)
$|\omega(G^{(i)})-\omega(G^{(j)})|\leq 1$ for any $i$ and $j$,
where $\omega(G^{(i)})$ is the clique number of $G^{(i)}$.

We also show that the Segre product of \cm\ graded rings with linear non-zerodivisor
which are \gor\ on the punctured spectrum is also \gor\ on the punctured spectrum
if all but one rings are standard graded.
\gor\ on the punctured spectrum
\end{abstract}

\section{Introduction}

In 1963, Bass \cite{bas} defined the notion of \gor\ rings and showed
ubiquity of \gor\ rings as the title of the paper indicates.
\gor\ property lies between \cm\ property and complete intersection property
and attracts many researchers.
However, during the progress of the research, researchers gradually felt that 
there is a rather long distance between \gor\ property and \cm\ property.
After that, many efforts are made to define a new notion between \gor\ and \cm\
properties in order to fill the gap between the two notions.

Two major successful notions of these are almost \gor\ and nearly
\gor\ properties.
Almost \gor\ property is first defined for 1-dimensional 
analytically unramified local rings by Barucci and Fr\"oberg \cite{bf},
next for general 1-dimensional local rings by Goto, Matsuoka and Phung \cite{gmp}
and finally by Goto, Takahashi and Taniguchi for general rings \cite{gtt}.
The present author gave a characterization of almost \gor\ property of Hibi rings,
the Ehrhart ring of the order polytope of a partially ordered set \cite{ag}.
Nearly \gor\ property is defined by Herzog, Hibi and Stamate \cite{hhs} by using
the trace of the canonical module of a local or graded ring.
They gave a characterization of nearly \gor\ property of Hibi rings
and the present author and Page gave a characterization of nearly \gor\
property of the Ehrhart ring of the chain polytope of a partially ordered set \cite{mp}.

In this paper, we investigate nearly \gor\ property of the Ehrhart ring of the 
stable set polytope of an h-perfect graph and characterize nearly \gor\ property completely.
Since perfect or t-perfect graphs are h-perfect, this characterization is also valid
for perfect or t-perfect graphs.
Hibi and Stamate asserted that they characterized when the Ehrhart ring of the stable
set polytope of a perfect graph is nearly \gor\ \cite[Theorem 13]{hs}.
However, their proof has a serious gap.
It seems like that they used an assertion which has a counter example.
See Remark \ref{rem:hs gap}.
Therefore, we do not use their results and prove the above mentioned results
independently.

In the way of studying nearly \gor\ property of the Ehrhart ring of an h-perfect
graph, we study the property ``\gor\ on the punctured spectrum''.
Since we are mainly interested in quotient rings of polynomial rings
which are quasi-unmixed, generalized \cm\ property is equivalent to
\cm\ on the punctured spectrum.
Since generalized \cm\ is an 
important notion, \gor\ on the punctured spectrum must be an important notion.
Therefore, we study the \gor\ on the punctured spectrum property for general graded rings.
See Theorem \ref{thm:gor pspec general}.

This paper is organized as follows.
In Section \ref{sec:pre}, we fix notation and recall known facts.
Next in Section \ref{sec:pspec}, we study the property \gor\ on the punctured spectrum
for Segre product of graded rings.
Almost all known results concerning this assumed that the rings involved are
standard graded.
We show that in almost all of them, one can replace standard graded property by the
existence of a linear non-zerodivisor.
Finally in Section \ref{sec:h-perfect}, we characterize the nearly \gor\ property
of the Ehrhart ring of an h-perfect graph by the status of connected components of
the graph.
Since perfect or t-perfect graphs are h-perfect, this characterization is also applicable
to perfect or t-perfect graphs.

By combining our previous results (Fact \ref{fact:gor cri}), we have characterized
nearly \gor\ property of the Ehrhart rings of h-perfect, perfect and t-perfect
graphs in the words of graph theory.

\section{Preliminaries}

\mylabel{sec:pre}

In this section, we establish notation and terminology.
In this paper, all rings and algebras are assumed to 
be commutative 
with an identity element.
Further, all graphs are finite simple graphs without loop.
We denote the set of nonnegative integers, 
the set of integers, 
the set of rational numbers and 
the set of real numbers
by $\NNN$, $\ZZZ$, $\QQQ$ and $\RRR$ respectively.

Let $\KKK$ be a field.
We call a finitely generated non-negatively graded $\KKK$-algebra
$\bigoplus_{n=0}^\infty R_n$ with $R_0=\KKK$ an $\NNN$-graded $\KKK$-algebra.
If an $\NNN$-graded $\KKK$-algebra is generated by degree 1 elements,
then we say that it is a standard graded $\KKK$-algebra.
For an $\NNN$-graded $\KKK$-algebra $R=\bigoplus_{n=0}^\infty R_n$,
we set $\mmmm_R\define\bigoplus_{n=1}^\infty R_n$ and for a
$\ZZZ$-graded $R$-module $M=\bigoplus_{n\in \ZZZ} M_n$ and an integer $k$,
we denote by $M_{\geq k}$ the $R$-submodule $\bigoplus_{n=k}^\infty M_n$ of $M$.
We denote by $a(R)$ the $a$-invariant of a \cm\ $\NNN$-graded $\KKK$-algebra $R$,
defined by Goto and Watanabe \cite{gw}.

For $\NNN$-graded $\KKK$-algebras $R^{(1)}$, \ldots, $R^{(t)}$, we define the
Segre product $R^{(1)}\#\cdots\#R^{(t)}$ 
of $R^{(1)}$, \ldots, $R^{(t)}$ by
$R^{(1)}\#\cdots\#R^{(t)}
=\bigoplus_{n=0}^\infty R^{(1)}_n\otimes_\KKK\cdots\otimes_\KKK R^{(t)}_n$.
Further, for $\ZZZ$-graded $R^{(i)}$-modules $M^{(i)}$ for $1\leq i\leq t$,
we define the Segre product $M^{(1)}\#\cdots\#M^{(t)}$ by
$M^{(1)}\#\cdots\#M^{(t)}
=\bigoplus_{n\in \ZZZ}M^{(1)}_n\otimes_\KKK\cdots\otimes_\KKK M^{(t)}_n$.
For elements $x_i\in M^{(i)}_n$ for $1\leq i\leq t$, we sometimes denote
the element 
$x_1\otimes\cdots\otimes x_t\in M^{(1)}_n\otimes\cdots\otimes M^{(t)}_n$
by $x_1\#\cdots \#x_t$
in order to emphasize that it is an element of the Segre product.

For a set $X$, we denote by $\#X$ the cardinality of $X$.
For sets $X$ and $Y$, we define $X\setminus Y\define\{x\in X\mid x\not\in Y\}$.
For nonempty sets $X$ and $Y$, we denote the set of maps from $X$ to $Y$ by $Y^X$.
If $X$ is a finite set, we identify $\RRR^X$ with the Euclidean space 
$\RRR^{\#X}$.
For $f$, $f_1$, $f_2\in\RRR^X$ and $a\in \RRR$,
we define
maps $f_1\pm f_2$ and $af$ 
by
$(f_1\pm f_2)(x)=f_1(x)\pm f_2(x)$ and
$(af)(x)=a(f(x))$
for $x\in X$.
Let $A$ be a subset of $X$.
We define the characteristic function $\chi_A\in\RRR^X$ by
$\chi_A(x)=1$ for $x\in A$ and $\chi_A(x)=0$ for $x\in X\setminus A$.
For a nonempty subset $\msXXX$ of $\RRR^X$, we denote by $\conv\msXXX$
(resp.\ $\aff\msXXX$)
the convex hull (resp.\ affine span) of $\msXXX$.

\begin{definition}
\mylabel{def:xi+}
\rm
Let $X$ be a finite set and $\xi\in\RRR^X$.
For $B\subset X$, we set $\xip(B)\define\sum_{b\in B}\xi(b)$.
We define the empty sum to be 0, i.e., if $B=\emptyset$, then $\xi^+(B)=0$.
\end{definition}

For unexplained terminology of graph theory, we consult \cite{die}.
A stable set of a graph $G=(V,E)$ is a subset $S$ of $V$ with no two
elements of $S$ are adjacent.
We treat the empty set as a stable set.

\begin{definition}
\rm
The stable set polytope $\stab(G)$ of a graph $G=(V,E)$ is
$$
\conv\{\chi_S\in\RRR^V\mid \mbox{$S$ is a stable set of $G$.}\}
$$
\end{definition}
It is clear that for $f\in\stab(G)$,

\begin{enumerate}
\item
\mylabel{item:nonneg}
$0\leq f(x)\leq 1$ for any $x\in V$.
\item
\mylabel{item:clique}
$f^+(K)\leq 1$ for any clique $K$ in $G$.
\item
\mylabel{item:cycle}
$f^+(C)\leq\frac{\#C-1}{2}$ for any odd cycle $C$.
\end{enumerate}

\begin{definition}
\rm
We set
\begin{eqnarray*}
\hstab(G)&\define&\{f\in \RRR^V\mid\mbox{$f$ satisfies \ref{item:nonneg}, 
\ref{item:clique} and \ref{item:cycle} above}\},\\
\qstab(G)&\define&\{f\in \RRR^V\mid\mbox{$f$ satisfies \ref{item:nonneg} 
and \ref{item:clique} above}\}\qquad\mbox{and}\\
\tstab(G)&\define&\left\{f\in \RRR^V\left|\
\vcenter{\hsize=.5\textwidth\relax
\noindent$f$ satisfies \ref{item:nonneg}
and \ref{item:cycle} above and $f^+(e)\leq 1$ for any $e\in E$}\right.\right\}.
\end{eqnarray*}
\end{definition}
Note that $\hstab(G)$, $\tstab(G)$ and $\qstab(G)$ are convex polytopes 
since it is bounded by \ref{item:nonneg}.
It is immediately seen that $\stab(G)\subset\hstab(G)=\qstab(G)\cap\tstab(G)$.
If $\stab(G)=\hstab(G)$ (resp.\ $\stab(G)=\tstab(G)$),
then $G$ is called an h-perfect (resp.\ t-perfect) graph.
Further, by \cite[Theorem 3.1]{chv}, $G$ is perfect if and only if
$\stab(G)=\qstab(G)$.
Thus, perfect graphs and t-perfect graphs are h-perfect.

We fix notation about Ehrhart rings.
Let $\KKK$ be a field, $X$ a finite set  
and $\msPPP$ a rational convex polytope in $\RRR^X$, i.e., 
a convex polytope whose vertices are contained in $\QQQ^X$.
Let $-\infty$ be a new element
with $-\infty\not\in X$ and
set $X^-\define X\cup\{-\infty\}$.
Also let $\{T_x\}_{x\in X^-}$ be
a family of indeterminates indexed by $X^-$.
For $f\in\ZZZ^{X^-}$, 
we denote the Laurent monomial 
$\prod_{x\in X^-}T_x^{f(x)}$ by $T^f$.
We set $\deg T_x=0$ for $x\in X$ and $\deg T_{-\infty}=1$.
Then the Ehrhart ring of $\msPPP$ over a field $\KKK$ is the $\NNN$-graded subring
$$
\KKK[T^f\mid f\in \ZZZ^{X^-}, f(-\infty)>0, \frac{1}{f(-\infty)}f|_X\in\msPPP]
$$
of the Laurent polynomial ring $\KKK[T_x^{\pm1}\mid x\in X^-]$,
where $f|_X$ is the restriction of $f$ to $X$.
We denote the Ehrhart ring of $\msPPP$ over $\KKK$ by $\ekp$.

It is known that $\ekp$ is Noetherian.
Therefore normal and \cm\ by the result of Hochster \cite{hoc}.
Further, 
by the description of the canonical module of a normal affine semigroup ring
by Stanley \cite[p.\ 82]{sta2}, we see 
the following.

\begin{lemma}
\mylabel{lem:sta desc}
The ideal 
$$\bigoplus_{f\in\ZZZ^{X^-}, f(-\infty)>0, \frac{1}{f(-\infty)}f|_X\in\relint\msPPP}
\KKK T^f
$$
of $\ekp$ is the canonical module of $\ekp$,
where $\relint\msPPP$ denotes the interior of $\msPPP$ in the topological space
$\aff\msPPP$.
\end{lemma}
We denote the ideal of the above lemma by $\omega_{\ekp}$ and call the canonical ideal of $\ekp$.

Let $X_1$, \ldots, $X_t$ be pairwise disjoint finite sets and let $\msPPP_i$ be a rational
convex polytope in $\RRR^{X_i}$ for each $i$.
Then it is easily verified that $\msPPP_1\times\cdots\times\msPPP_t$ is a rational
convex polytope in $\RRR^{X_1\cup\cdots\cup X_t}$ and
$E_\KKK[\msPPP_1\times\cdots\times \msPPP_t]=E_\KKK[\msPPP_1]\#\cdots\# E_\KKK[\msPPP_t]$.


Let $R$ be a Noetherian normal domain.
Then the set of divisorial ideals $\Div(R)$ form a group by the operation
$I\cdot J\define R:_{Q(R)}(R:_{Q(R)} IJ)$ for $I$, $J\in\Div(R)$,
where $Q(R)$ is the quotient field of $R$.
See e.g., \cite[Chapter I]{fos} for details.
We denote the $n$-th power of $I\in\Div(R)$ in this group by $I^{(n)}$
and call the $n$-th symbolic power of $I$ for any integer $n$.
Note that if $R$ is a \cm\ local or graded ring over a field with canonical
module $\omega$, then $\omega$ is isomorphic to 
a divisorial ideal.
In particular, $\omega_{\ekp}\in\Div(\ekp)$ for a rational convex polytope $\msPPP$.
See e.g., \cite[Chapter 3]{bh} for details.

Next we recall the results of \cite{hhs}.
First we recall the following.

\begin{definition}
\rm
Let $R$ be a ring and $M$ an $R$-module.
We set 
$$
\trace(M)\define\sum_{\varphi\in\hom(M,R)}\varphi(M)
$$
and call $\trace(M)$ the trace of $M$.
\end{definition}
%
%
%
We recall the following.

\begin{fact}[{\cite[Lemma 1.1]{hhs}}]
\mylabel{fact:hhs1.1}
Let $R$ be a ring and $I$ an ideal of $R$ containing an $R$-regular element.
Also let $Q(R)$ be the total quotient ring of fractions of $R$
and set $I^{-1}\define\{x\in Q(R)\mid xI\subset R\}$.
Then
$$
\trace(I)=I^{-1}I.
$$
\end{fact}
Note that if $R$ is a Noetherian normal domain and 
$I$ is a divisorial ideal, then $I^{-1}=I^{(-1)}$.
Moreover, we recall the following.

\begin{fact}[{\cite[Lemma 2.1]{hhs}}]
\mylabel{fact:hhs2.1}
Let $R$ be a \cm\ local or graded ring over a field with canonical
module $\omega_R$.
Then for $\pppp\in\spec(R)$,
$$
R_\pppp\mbox{ is \gor}\iff
\pppp\not\supset\trace(\omega_R).
$$
\end{fact}
On account of this fact, Herzog, Hibi and Stamate \cite[Definition 2.2]{hhs}
defined the nearly \gor\ property.

\begin{definition}
\rm
Let $R$ be a \cm\ local or graded ring over a field with canonical
module $\omega_R$.
If $\trace(\omega_R)\supset \mmmm_R$, then $R$ is called a nearly \gor\ ring.
\end{definition}

Now we recall our previous results \cite{ghp}.

\begin{definition}
\rm
\mylabel{def:un}
Set 
$\msKKK=\msKKK(G)\define
\{K\subset V\mid
K$ is a clique of $G$ and size of $K$ is less than or equal to 3$\}$.
For $n\in\ZZZ$.
We set
$$
\UUUUUn=
\UUUUUn(G)
\define\left\{\mu\in\ZZZ^{V^-}\left|\ 
\vcenter{\hsize=.5\textwidth\relax\noindent
$\mu(z)\geq n$ for any $z\in V$,
$\mu^+(K)\leq \mu(-\infty)-n$ for any maximal clique $K$ of $G$ and
$\mu^+(C)\leq\mu(-\infty)\frac{\#C-1}{2}-n$ for any odd cycle $C$ without
chord and length at least $5$}\right.\right\},
$$
$$
\tUUUUU^{(n)}=
\tUUUUU^{(n)}(G)
\define\left\{\mu\in\ZZZ^{V^-}\left|\ 
\vcenter{\hsize=.5\textwidth\relax\noindent
$\mu(z)\geq n$ for any $z\in V$,
$\mu^+(K)\leq \mu(-\infty)-n$ for any maximal element $K$ of $\msKKK$ and
$\mu^+(C)\leq\mu(-\infty)\frac{\#C-1}{2}-n$ for any odd cycle $C$ without
chord and length at least $5$}\right.\right\}
$$
and
$$
\qUUUUU^{(n)}=
\qUUUUU^{(n)}(G)
\define\left\{\mu\in\ZZZ^{V^-}\left|\
\vcenter{\hsize=.5\textwidth\relax\noindent
$\mu(z)\geq n$ for any $z\in V$ and
$\mu^+(K)\leq \mu(-\infty)-n$ for any maximal clique of $G$}\right.\right\}.
$$
\end{definition}
By this notation, the following holds.

\begin{fact}[{\cite[Proposition 3.7 and Remark 3.10]{ghp}}]
\mylabel{fact:symb power}
$$\omega_\ekhsg^{(n)}=\bigoplus_{\mu\in\UUUUU^{(n)}}\KKK T^\mu,
$$
$$
\omega_{\ektsg}^{(n)}=\bigoplus_{\mu\in\tUUUUU^{(n)}}\KKK T^\mu
$$
and
$$
\omega_{\ekqsg}^{(n)}=\bigoplus_{\mu\in\qUUUUU^{(n)}}\KKK T^\mu
$$
for any $n\in\ZZZ$.
\end{fact}
Further, we state the following.

\begin{definition}
\rm
For a graph $G$, we denote by $\omega(G)$ the maximum size of cliques of $G$.
If all maximal cliques (resp.\ maximal elements of $\msKKK(G)$) have the same
size, we say that $G$ is pure (resp.\ t-pure).
\end{definition}
Let $\eta$ be an element of $\ZZZ^{V^-}$ with $\eta(x)=1$
for $x\in V$ and $\eta(-\infty)=\omega(G)+1$ (resp.\
$\eta(-\infty)=\omega(G)+1$,  
$\eta(-\infty)=\min\{\omega(G),3\}+1$).
Then it is easily verified by Fact \ref{fact:symb power}
that $T^\eta$ is an element of $\omega_{\ekhsg}$
(resp.\ $\omega_{\ekqsg}$, $\omega_{\ektsg}$) of minimum degree.
Therefore,
we see the following.

\begin{lemma}
\mylabel{lem:a-inv}
$$
a(\ekhsg)=a(\ekqsg)=-\omega(G)-1
$$
and
$$
a(\ektsg)=-\min\{\omega(G),3\}-1.
$$
\end{lemma}
We also showed the following.

\begin{fact}[{\cite[Theorem 3.8 and Remark 3.10]{ghp}}]
\mylabel{fact:gor cri}

\noindent
\begin{enumerate}
\item
$\ekhsg$ is \gor\ if and only if
\begin{enumerate}
\item
\mylabel{item:clique size}
$G$ is pure and
\item
\mylabel{item:3cond}
\begin{enumerate}
\item
\mylabel{item:1}
$\omega(G)=1$, 
\item
\mylabel{item:2}
$\omega(G)=2$ and there is no odd cycle without chord and length at least 7 or
\item
\mylabel{item:3}
$\omega(G)\geq 3$ and there is no odd cycle without chord and length at least 5.
\end{enumerate}
\end{enumerate}
\item
$\ektsg$ is \gor\ if and only if
\begin{enumerate}
\item
$E=\emptyset$,
\item
$G$ has no isolated vertex nor triangle and there is no odd cycle
without chord and length at least 7 or
\item
all maximal cliques of $G$ have size at least 3
and there is no odd cycle without chord and length at least 5.
\end{enumerate}
\item
$\ekqsg$ is \gor\ if and only if
$G$ is pure.
\end{enumerate}
\end{fact}

The following lemma is very easily proved but very useful.

\begin{lemma}
\mylabel{lem:triv}
Suppose that $\eta\in\UUUUU^{(1)}$ (resp.\ $\tUUUUU^{(1)}$, $\qUUUUU^{(1)}$)
and $\zeta\in\UUUUU^{(-1)}$ (resp.\ $\tUUUUU^{(-1)}$, $\qUUUUU^{(-1)}$).
If $x\in V$ and $(\eta+\zeta)(x)=0$, then $\eta(x)=1$ and $\zeta(x)=-1$.
\end{lemma}


\section{$\NNN$-graded $\KKK$-algebras which are \gor\ on the punctured
spectrum}
\mylabel{sec:pspec}

In this section, we study the \gor\ property on the punctured spectrum
of the Segre product of several $\NNN$-graded $\KKK$-algebras.
Although the results of the next section can be proved by weaker result
than stated in this section, we prove it in general form, since the \gor\
on the punctured spectrum  is an important property.

We first state the following.

\begin{definition}
\rm
Let $R$ be a \cm\ $\NNN$-graded $\KKK$-algebra.
If $R_\pppp$ is \gor\ for any $\pppp\in\spec(R)\setminus\{\mmmm_R\}$,
we say that $R$ is \gor\ on the punctured spectrum.
\end{definition}
By Fact \ref{fact:hhs2.1}, we see that $R$ is \gor\ on the punctured spectrum
if and only if $\trace(\omega_R)\supset\mmmm_R^n$ for $n\gg 0$.

Next we recall the following result.

\begin{fact}[{\cite[Lemma (4.4.1)]{gw}}]
\mylabel{fact:gw 4.4.1}
Let $R$ be an $\NNN$-graded $\KKK$-algebra of dimension $r$ with
linear $R$-regular element.
Then $(\underline{\mathrm{H}}_\mmmm^r(R))_n\neq 0$ for $n\leq a(R)$.
In particular, if $R$ is \cm, then $(\omega_R)_n\neq 0$ for $n\geq -a(R)$.
\end{fact}
It is assumed that $R$ is standard graded in \cite[Lemma (4.4.1)]{gw}.
However, the standardness is only used to ensure the existence of linear
$R$-regular element.
Further, the final assertion can be proved by the following way.

Let $x$ be a linear $R$-regular element of $R$.
Then, since $\omega_R$ is a maximal \cm\ $R$-module,
$x$ is also an $\omega_R$-regular element (see e.g., \cite[Theorem 17.2]{mat}).
Therefore, multiplication of $x$ induces an injection 
$\omega_R(-1)\to\omega_R$.
Therefore, $\dim_\KKK(\omega_R)_n\leq \dim_\KKK(\omega_R)_{n+1}$ for any $n\in\ZZZ$.
Since $(\omega_R)_{-a(R)}\neq0$, we see that $(\omega_R)_n\neq0$ for any $n\in\ZZZ$
with $n\geq -a(R)$.
 
Next we improve \cite[Proposition 2.2]{hmp}.
In \cite[Proposition 2.2]{hmp}, it is assumed that all rings involved are standard graded.

\begin{prop}
\mylabel{prop:hmp 2.2}
Let $R^{(i)}$ be a \cm\ $\NNN$-graded $\KKK$-algebra which has a linear
$R^{(i)}$-regular element, $\dim R^{(i)}\geq 2$ and $a(R^{(i)})<0$ for
$i=1$, \ldots, $t$.
Then
$$
R=R^{(1)}\#\cdots\# R^{(t)}
$$
is \cm, $\omega_R=\omega_{R^{(1)}}\#\cdots\#\omega_{R^{(t)}}$,
$a(R)=\min\{a(R^{(i)})\mid 1\leq i\leq t\}$, 
$\dim R=\sum_{i=1}^t \dim R^{(i)}-t+1$
and $R$ has a linear $R$-regular element.
\end{prop}
\begin{proof}
We prove by induction on $t$.
The case where $t=1$ is trivial.
Suppose that $t>1$ and set $S=R^{(2)}\#\cdots\# R^{(t)}$.

By induction hypothesis, we see that $S$ is \cm, 
$\omega_S=\omega_{R^{(2)}}\#\cdots\#\omega_{R^{(t)}}$,
$\dim S=\sum_{i=2}^t\dim R^{(i)}-t+2$,
$a(S)=\min\{a(R^{(i)})\mid 2\leq i\leq t\}$ 
and $S$ has a linear $S$-regular element.
Since $\dim S\geq 2(t-1)-t+2=t\geq 2$,
we see by \cite[Theorem (4.2.3)]{gw} that 
$R=R^{(1)}\#S$ is \cm,
$\dim R=\dim R^{(1)}+\dim S-1=\sum_{i=1}^t\dim R^{(i)}-t+1$.
Further, by \cite[Theorem (4.3.1)]{gw}, we see that
$\omega_R=\omega_{R^{(1)}}\#\omega_S=\omega_{R^{(1)}}\#\cdots\#\omega_{R^{(t)}}$.
Moreover, by Fact \ref{fact:gw 4.4.1}, we see that $(\omega_S)_n\neq 0$
for $n\geq -a(S)$ and $(\omega_{R^{(1)}})_n\neq 0$ for $n\geq -a(R^{(1)})$.
Thus, $(\omega_R)_n\neq 0$ if and only if $n\geq \max\{-a(R^{(1)}), -a(S)\}
=-\min\{a(R^{(1)}), a(S)\}$
and therefore $a(R)=\min\{a(R^{(1)}), a(S)\}=\min\{a(R^{(i)})\mid 1\leq i\leq t\}$.

Finally, let $x$ (resp.\ $y$) be a linear $R^{(1)}$-regular (resp.\ $S$-regular)
element of $R^{(1)}$ (resp.\ $S$).
We show that $x\#y$ is a linear $R^{(1)}\#S$-regular element.
It is enough to show that for any homogeneous element $\alpha$ of 
$R^{(1)}\#S$ with $\alpha\neq 0$, $(x\#y)\alpha\neq0$.
Let $\deg\alpha=d$ and write $\alpha=\sum_{j=1}^\ell z_j\#w_j$,
where $z_j\in R^{(1)}_d$ (resp.\ $w_j\in S_d$), 
$z_1$, \ldots, $z_\ell$ are linearly independent over $\KKK$ and
$w_j\neq 0$ for any $1\leq j\leq \ell$.
Then $(x\#y)\alpha=\sum_{j=1}^\ell xz_j\#yw_j$, $xz_1$, \ldots, $xz_\ell$
are linearly independent over $\KKK$ and $yw_j\neq0$ for any $1\leq j\leq \ell$,
since $x$ (resp.\ $y$) is an $R^{(1)}$-regular (resp.\ $S$-regular) element.
Therefore, $(x\#y)\alpha=\sum_{j=1}^\ell xz_j\#yw_j\neq 0$.
\end{proof}

Next we show the following.

\begin{lemma}
\mylabel{lem:trace elem}
Let $R$ be an $\NNN$-graded $\KKK$-algebra and $M$ a finitely generated graded
$R$-module.
Suppose that $\hom_R(M,R)$ is generated by homogeneous elements $\varphi_1$, \ldots,
$\varphi_\ell$.
Set $\deg\varphi_i=d_i$ for $1\leq i\leq \ell$.
Then for any homogeneous element $x\in\trace(M)$ of degree $d$, there are
homogeneous elements $\alpha_1$, \ldots, $\alpha_\ell\in M$ of
degree $d-d_1$, \ldots, $d-d_\ell$ respectively
(we define that the 0 element has arbitrary degree as in \cite[Section 1.5]{bh})
such that $x=\sum_{j=1}^\ell\varphi_j(\alpha_j)$.
\end{lemma}
\begin{proof}
Since $x\in \trace(M)$, there are homogeneous elements $\psi_1$, \ldots, $\psi_m\in\hom_R(M,R)$
and homogeneous elements $\gamma_1$, \ldots, $\gamma_m\in M$ with $\deg\gamma_i=d-\deg\psi_i$
for $1\leq i\leq m$ such that
$x=\sum_{i=1}^m\psi_i(\gamma_i)$.
Set $\psi_i=\sum_{j=1}^\ell \beta_{ij}\varphi_j$ for $1\leq i\leq m$,
where $\beta_{ij}$ is a homogeneous element of $R$ with degree $\deg\psi_i-d_j$.
Then
$$
x=\sum_{i=1}^m\psi_i(\gamma_i)
=\sum_{i=1}^m(\sum_{j=1}^\ell \beta_{ij}\varphi_j)(\gamma_i)
=\sum_{j=1}^\ell \varphi_j(\sum_{i=1}^m\beta_{ij}\gamma_i)
$$
and $\beta_{ij}\gamma_i$ is a homogeneous element of $M$ with degree $d-d_j$ for any $i$.
Therefore, it is enough to set $\alpha_j=\sum_{i=1}^m\beta_{ij}\gamma_i$ for $1\leq j\leq \ell$.
\end{proof}
The following fact is easily proved.

\begin{lemma}
\mylabel{lem:stand segre}
Let $R^{(1)}$, \ldots, $R^{(t)}$ be standard graded $\KKK$-algebras.
Then 
$$
R^{(1)}\#\cdots\#R^{(t)}
$$
is also a standard graded $\KKK$-algebra.
\end{lemma}

Now we prove the following.

\begin{thm}
\mylabel{thm:gor pspec general}
Let $R^{(1)}$, \ldots, $R^{(t)}$ be \cm\ $\NNN$-graded $\KKK$-algebras \gor\
on the punctured spectrum and $\dim R^{(i)}\geq 2$, $a(R^{(i)})<0$ and
$R^{(i)}$ has a linear $R^{(i)}$-regular element for every $i$ with $1\leq i\leq t$.
If at most one of $R^{(1)}$, \ldots, $R^{(t)}$ is not standard graded,
then
$$
R=R^{(1)}\#\cdots\#R^{(t)}
$$
is \gor\ on the punctured spectrum.
\end{thm}
\begin{proof}
We prove by induction on $t$.
The case where $t=1$ is trivial.
We prove the case where $t=2$ later.
If $t\geq 3$, set $S=R^{(2)}\#\cdots\#R^{(t)}$.
Then by Proposition \ref{prop:hmp 2.2}, we see that $S$ is a \cm\ $\NNN$-graded 
$\KKK$-algebra with $\dim S\geq 2$, $a(S)<0$ and has a linear $S$-regular element.
Further, by induction hypothesis, $S$ is \gor\ on the punctured spectrum.
If $R^{(1)}$ is not standard graded, then by Lemma \ref{lem:stand segre}, we see 
that $S$ is a standard graded $\KKK$-algebra.
Thus, at most one of $R^{(1)}$ or $S$ is not standard graded.
Therefore, by the case where $t=2$, we see that $R=R^{(1)}\#S$ is \gor\ on the
punctured spectrum.

Now we prove the case where $t=2$.
Note by Proposition \ref{prop:hmp 2.2}, $\omega_R=\omega_{R^{(1)}}\#\omega_{R^{(2)}}$.

By symmetry, we may assume that $R^{(2)}$ is standard graded.
Take a positive integer $d$ with $\trace(\omega_{R^{(i)}})\supset R_{\geq d}^{(i)}$ for $i=1$, $2$.
Further, take homogeneous generators $\varphi_{i1}$, \ldots, $\varphi_{is_i}$ of
$\hom_{R^{(i)}}(\omega_{R^{(i)}},R^{(i)})$ for $i=1$, $2$.
Set $\deg \varphi_{ij}=d'_{ij}$ for any $i$ and $j$ and 
$d'=\max\{|d'_{1j}|+|d'_{2\jmath'}|\mid 1\leq j\leq s_1, 1\leq \jmath'\leq s_2\}$.
Take homogeneous generators $z_1$, \ldots, $z_u$ of the ideal $R^{(1)}_{\geq d+d'}$,
set $\deg z_i=d''_i$ for $1\leq i\leq u$ and
$d''=\max\{d''_1,\ldots, d''_u\}$.

Let $N$ be an arbitrary integer with $N\geq d+d'+d''$.
We show that $\trace(\omega_R)\supset R_N$.

Let $x_i$ be an arbitrary element of $R^{(i)}_N$ for $i=1$, $2$.
Since $x_1\in R^{(1)}_N\subset R^{(1)}_{\geq d+d'}$, 
there are homogeneous elements $w_1$, \ldots, $w_u$ of $R$ 
with degree $N-d''_1$, \ldots, $N-d''_u$ respectively
and
$$
x_1=w_1z_1+\cdots+w_uz_u.
$$
Then, since $\deg z_k\leq d''$, we see that $\deg w_k\geq d+d'$.
In particular, $w_k\in R^{(1)}_{\geq d+d'}\subset \trace(\omega_{R^{(1)}})$ for $1\leq k\leq u$.
Therefore, by Lemma \ref{lem:trace elem}, we can write 
$$
w_k=\sum_{j=1}^{s_1}\varphi_{1j}(\alpha_{jk})
\qquad\mbox{ for $1\leq k\leq u$.}
$$
Then 
$$
x_1=\sum_{k=1}^u\sum_{j=1}^{s_1}\varphi_{1j}(\alpha_{jk})z_k
=\sum_{k=1}^u\sum_{j=1}^{s_1}(z_k\varphi_{1j})(\alpha_{jk}).
$$
Thus, it is enough to show that
$$
(z_k\varphi_{jk})(\alpha_{jk})\otimes x_2\in\trace(\omega_R)
$$
for any $k$ and $j$.

We fix $k$ and $j$ with $1\leq k\leq u$ and $1\leq j\leq s_1$.
Since $R^{(2)}$ is standard graded, $x_2\in R^{(2)}_N$ and $N>d$, we can write
$x_2$ as
$$
x_2=f_1y_1+\cdots +f_vy_v,
$$
where $y_\ell\in R^{(2)}_d$ and $f_\ell\in R^{(2)}_{N-d}$ for $1\leq \ell\leq v$.
Further, since $y_\ell\in R^{(2)}_d\subset\trace(\omega_{R^{(2)}})$, 
we see by Lemma \ref{lem:trace elem} that there are homogeneous elements
$\beta_{1\ell}$, \ldots, $\beta_{s_2\ell}$ of degree $d-d'_{21}$, \ldots, $d-d'_{2s_2}$
respectively with
$$
y_\ell=\sum_{\jmath'=1}^{s_2}\varphi_{2\jmath'}(\beta_{\jmath'\ell}).
$$
Since 
$$
x_2=\sum_{\ell=1}^v\sum_{\jmath'=1}^{s_2}\varphi_{2\jmath'}(\beta_{\jmath'\ell})f_\ell,
$$
it is enough to show that
$$
(z_k\varphi_{1j})(\alpha_{jk})\otimes \varphi_{2\jmath'}(\beta_{\jmath'\ell})f_\ell
\in\trace(\omega_R)
$$
for any $\ell$ and $\jmath'$ with $1\leq \ell\leq v$ and $1\leq \jmath'\leq s_2$.

We fix $\ell$ with $1\leq \ell\leq v$ and $\jmath'$ with $1\leq \jmath'\leq s_2$ 
and set 
$$
d'''=\deg(z_k\varphi_{1j})-\deg\varphi_{2\jmath'}=d''_k+d'_{1j}-d'_{2\jmath'}.
$$
Then
$$
d'''\geq d+d'+d'_{1j}-d'_{2\jmath'}\geq d,
$$
since $z_k\in R^{(1)}_{\geq d+d'}$ and
$$
N-d-d'''\geq d+d'+d''-d-d''_k-d'_{1j}+d'_{2\jmath'}\geq 0.
$$
Therefore, we can write 
$$
f_\ell=g_1h_1+\cdots+g_wh_w,
$$
where $g_{k'}\in R^{(2)}_{N-d-d'''}$ and $h_{k'}\in R^{(2)}_{d'''}$
for any $k'$ with $1\leq k'\leq w$,
since $R^{(2)}$ is standard graded and $f_\ell\in R^{(2)}_{N-d}$.
Then
$$
\varphi_{2\jmath'}(\beta_{\jmath'\ell})f_\ell=
\sum_{k'=1}^w \varphi_{2\jmath'}(\beta_{\jmath'\ell})g_{k'}h_{k'}
=\sum_{k'=1}^w (h_{k'}\varphi_{2\jmath'})(g_{k'}\beta_{\jmath'\ell}).
$$
Since
$$
\deg(h_{k'}\varphi_{2\jmath'})=d'''+\deg\varphi_{2\jmath'}=\deg(z_k\varphi_{1j}),
$$
we see that the map
$$
z_k\varphi_{1j}\otimes h_{k'}\varphi_{2\jmath'}\colon
\omega_{R^{(1)}}\otimes\omega_{R^(2)}\to R^{(1)}\otimes R^{(2)}
$$
is an element of
$
\hom_{R^{(1)}\#R^{(2)}}(
\omega_{R^{(1)}}\#\omega_{R^(2)},R^{(1)}\# R^{(2)})
=\hom_R(\omega_R,R)
$.
Further, since
$$
\deg \alpha_{jk}=N-\deg(z_k\varphi_{1j})=N-\deg(h_{k'}\varphi_{2\jmath'})
=\deg(g_{k'}\beta_{\jmath'\ell}),
$$
we see that
$$
(z_k\varphi_{1j})(\alpha_{jk})\otimes(h_{k'}\varphi_{2\jmath'})(g_{k'}\beta_{\jmath'\ell})
=((z_k\varphi_{1j})\#(h_{k'}\varphi_{2\jmath'}))(\alpha_{jk}\#g_{k'}\beta_{\jmath'\ell})
\in\trace(\omega_R)
$$
for any $k'$.
Therefore,
$$
(z_k\varphi_{1j})(\alpha_{jk})\otimes 
\varphi_{2\jmath'}(\beta_{\jmath'\ell})f_\ell
=\sum_{k'=1}^w(z_k\varphi_{1j})(\alpha_{jk})\otimes(h_{k'}\varphi_{2\jmath'})
(g_{k'}\beta_{\jmath'\ell})
\in\trace(\omega_R).
$$

Thus, we have shown that $R_N\subset\trace(\omega_R)$.
Since $N$ is an arbitrary integer with $N\geq d+d'+d''$, we see that
$R_{\geq d+d'+d''}\subset\trace(\omega_R)$.
Since $R_{\geq d+d'+d''}$ is an $\mmmm_R$-primary ideal, we see that
$R$ is \gor\ on the punctured spectrum.
\end{proof}


\section{Nearly \gor\ property of the Ehrhart rings of the stable set polytopes
of h-perfect, perfect and t-perfect graphs}
\mylabel{sec:h-perfect}

In this section, we state a characterization of nearly \gor\ property of the 
Ehrhart rings of the stable set polytopes of h-perfect, perfect and t-perfect graphs.
We begin with the following.

\begin{lemma}
\mylabel{lem:non pure non gor}
Let $G=(V,E)$ be a graph and set
$R=\ekhsg$ (resp.\ $\ekqsg$, $\ektsg$).
If there are maximal cliques (resp.\
maximal cliques, maximal elements of $\msKKK$) $K_1$ and $K_2$
with $\#K_1>\#K_2$, then
$\trace(\omega_R)\subset R_{\geq\#K_1-\#K_2}$.
\end{lemma}
\begin{proof}
Let $T^\mu$ be an arbitrary monomial of $\trace(\omega_R)$, where
$\mu\in\ZZZ^{V^-}$.
Since $T^\mu\in\trace(\omega_R)=\omega_R\omega_R^{-1}$,
we see that there are $\eta\in\UUUUU^{(1)}$ (resp.\ $\qUUUUU^{(1)}$, $\tUUUUU^{(1)}$)
and $\zeta\in\UUUUU^{(-1)}$ (resp.\ $\qUUUUU^{(-1)}$, $\tUUUUU^{(-1)}$)
with $\mu=\eta+\zeta$.
Since $\eta\in\UUUUU^{(1)}$ (resp.\ $\qUUUUU^{(1)}$, $\tUUUUU^{(1)}$),
we see that $\eta^+(K_1)\geq \#K_1$.
Thus, $\eta(-\infty)\geq\eta^+(K_1)+1\geq\#K_1+1$.
On the other hand, since $\zeta\in \UUUUU^{(-1)}$ (resp.\ $\qUUUUU^{(-1)}$, $\tUUUUU^{(-1)}$),
we see that $\zeta^+(K_2)\geq-\#K_2$.
Thus, $\zeta(-\infty)\geq\zeta^+(K_2)-1\geq-\#K_2-1$.
Therefore,
$$
\mu(-\infty)=\eta(-\infty)+\zeta(-\infty)\geq\#K_1-\#K_2.
$$
This means that $T^\mu\in R_{\geq \#K_1-\#K_2}$.
\end{proof}

Next we state a property of a connected component which is not pure.

\begin{lemma}
\mylabel{lem:non pure connected}
Let $G=(V,E)$ be a graph and $G'$ a connected component of $G$.
If $G'$ is not pure (resp.\ t-pure), then there are maximal cliques
(resp.\ maximal elements of $\msKKK$) $K_1$ and $K_2$ in $G'$ such that
$\#K_1\neq\#K_2$ and $K_1\cap K_2\neq\emptyset$.
\end{lemma}
\begin{proof}
Take maximal cliques (resp.\ maximal elements of $\msKKK$) $L$ and $L'$ with
$\#L\neq\#L'$, $v\in L$ and $v'\in L'$.
Since $G'$ is connected, there are $v_1$, \ldots, $v_{n-1}\in V$ such that
$v_{i-1}$ and $v_{i}$ are adjacent for any $i=1$, \ldots, $n$, where we set
$v_0=v$ and $v_n=v'$.
Take a maximal clique (resp.\ maximal elment of $\msKKK$) $L_i$ with 
$L_i\supset\{v_{i-1}, v_i\}$ for $i=1$, \ldots, $n$.
If there is $i$ with $\#L_{i-1}\neq\#L_{i}$, then it is enough to set
$K_1=L_{i-1}$ and $K_2=L_i$, since $v_{i-1}\in L_{i-1}\cap L_i$.
If $\#L_{i-1}=\#L_i$ for any $i$, then $\#L\neq \#L_1$ or $\#L'\neq \#L_n$,
since $\#L\neq\#L'$.
If $\#L\neq\#L_1$, then it is enough to set $K_1=L$ and $K_2=L_1$, since
$v_0=v\in L\cap L_1$.
The case where $\#L'\neq \#L_n$ is proved similarly.
\end{proof}

In the next two lemmas, we state necessary conditions for 
$\ekhsg$ (resp.\ $\ekqsg$, $\ektsg$) to be \gor\ on the punctured spectrum.

\begin{lemma}
\mylabel{lem:gor pspec 1}
Let $G=(V,E)$ be a graph and set
$R=\ekhsg$ (resp.\ $\ekqsg$, $\ektsg$).
If $R$ is \gor\ on the punctured spectrum, then every connected component of
$G$ is pure (resp.\ pure, t-pure).
\end{lemma}
\begin{proof}
Assume the contrary and suppose that there is a connected component $G'$ of $G$ 
which is not pure (resp.\ pure, t-pure).
Since $R$ is \gor\ on the punctured spectrum, we can take a positive integer $n$
with $\trace(\omega_R)\supset R_{\geq n}$.

By Lemma \ref{lem:non pure connected}, we see that there are maximal cliques
(resp.\ maximal cliques, maximal elements of $\msKKK$) $K_1$ and $K_2$ in $G'$
with $\#K_1\neq \#K_2$ and $K_1\cap K_2\neq \emptyset$.
We may assume that $\#K_1>\#K_2$.
Take $p\in K_1\cap K_2$ and define $\mu\in\ZZZ^{V^-}$ by
$$
\mu(x)=
\left\{
\begin{array}{ll}
n&\qquad\mbox{$x=p$, $-\infty$},\\
0&\qquad\mbox{otherwise}.
\end{array}
\right.
$$
Then $\mu\in\UUUUU^{(0)}$ (resp.\ $\qUUUUU^{(0)}$, $\tUUUUU^{(0)}$)
and therefore $T^\mu\in R$.
Further, $T^\mu\in R_{\geq n}$ since $\mu(-\infty)=n$.

Since $T^\mu\in R_{\geq n}\subset\trace(\omega_R)=\omega_R\omega_R^{(-1)}$,
we see that there are $\eta\in\UUUUU^{(1)}$ (resp.\ $\qUUUUU^{(1)}$, $\tUUUUU^{(1)}$)
and $\zeta\in\UUUUU^{(-1)}$ (resp.\ $\qUUUUU^{(-1)}$, $\tUUUUU^{(-1)}$)
with $\mu=\eta+\zeta$.
By Lemma \ref{lem:triv}, we see that $\eta(x)=1$ and $\zeta(x)=-1$ for any 
$x\in (K_1\cup K_2)\setminus\{p\}$.
Therefore,
$\eta^+(K_1)=\#K_1+\eta(p)-1$
and
$\zeta^+(K_2)=-\#K_2+\zeta(p)+1$.
Thus,
$\eta(-\infty)\geq\eta^+(K_1)+1=\#K_1+\eta(p)$
and
$\zeta(-\infty)\geq\zeta^+(K_2)-1=-\#K_2+\zeta(p)$.

Since $\eta(p)+\zeta(p)=\mu(p)=n$, we see that
$\eta(-\infty)+\zeta(-\infty)\geq \#K_1-\#K_2+n>n$.
This contradicts to the fact that
$\eta(-\infty)+\zeta(-\infty)=\mu(-\infty)=n$.
\end{proof}
\begin{remark}
\rm
\mylabel{rem:hs gap}
Hibi and Stamate \cite[Lemma 11]{hs} asserted that they proved this result
in the case where $G$ is a perfect graph and $\eksg$ is \gor\ on the punctured spectrum.
\cite[Lemma 11]{hs} is used in the
proof of \cite[Theorem 13]{hs}, the main theorem of the latter half of their paper.
However they used in the proof the following assertion which has a counter example:
let $G'$ be a connected component of a graph $G$ which is not pure.
Then there is an edge $\{i_0,j_0\}$ in $G'$ such that $i_0$ is contained in a clique
with size $\omega(G')$ and there is no clique containing $j_0$ with size $\omega(G')$.

Consider the following graph $G$.
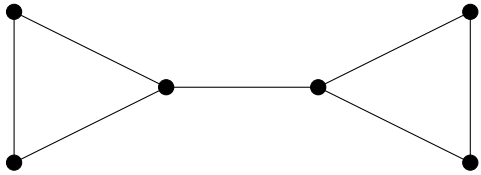
\begin{center}
\begin{tikzpicture}

\coordinate(X) at (0,2);
\coordinate(Y) at (0,0);
\coordinate(Z) at (2,1);
\coordinate(W) at (4,1);
\coordinate(U) at (6,2);
\coordinate(V) at (6,0);

\draw (Z)--(X)--(Y)--(Z)--(W)--(U)--(V)--(W);

\draw[fill] (X) circle [radius=0.1];
\draw[fill] (Y) circle [radius=0.1];
\draw[fill] (Z) circle [radius=0.1];
\draw[fill] (W) circle [radius=0.1];
\draw[fill] (U) circle [radius=0.1];
\draw[fill] (V) circle [radius=0.1];

\end{tikzpicture}
\end{center}
$G$ is connected, not pure and all the vertices are contained in a clique
with size $3=\omega(G)$.
Thus, there is no edge $\{i_0, j_0\}$ satisfying the above condition.
\end{remark}

\begin{lemma}
\mylabel{lem:gor pspec 2}
Let $G=(V,E)$ be a graph and set
$R=\ekhsg$ (resp.\ $\ektsg$).
If $R$ is \gor\ on the punctured spectrum, then for every connected component 
$G'$ of $G$, there is no odd cycle without chord and length longer than 3 in
$G'$ or the size of every maximal clique in $G'$ is 2 and there is no odd 
cycle without chord and length longer than 5.
\end{lemma}
\begin{proof}
Suppose that there is an odd cycle $C$  without chord and length longer than 3 in $G'$.
We show that the length of $C$ is 5 and every maximal clique in $G'$ has size 2.

Set $C=\{v_0, v_1, \ldots, v_{2\ell}\}$, where $\ell$ is an integer with $\ell\geq 2$,
$\{v_{i-1},v_i\}\in E$ for $1\leq i\leq 2\ell$ and $\{v_{2\ell},v_0\}\in E$.
Since $R$ is \gor\ on the punctured spectrum, there is a positive integer $n$ 
with $\trace(\omega_R)\supset R_{\geq n}$.
Define $\mu\in\ZZZ^{V^-}$ by
$$
\mu(x)=
\left\{
\begin{array}{ll}
n&\qquad\mbox{$x\in\{v_1, v_3, \ldots, v_{2\ell-1}, -\infty\}$},\\
0&\qquad\mbox{otherwise}.
\end{array}
\right.
$$
Then $\mu\in\UUUUU^{(0)}$ (resp.\ $\tUUUUU^{(0)}$), since $C$ does not have a chord.
Let $m$ be the maximum size of cliques in $G'$ (resp.\ elements of $\msKKK(G')$).
Note that every maximal clique in $G'$ (resp.\ maximal element of $\msKKK(G')$) has
size $m$ by Lemma \ref{lem:gor pspec 1}.
Since $T^\mu\in R_{\geq n}\subset\trace(\omega_R)$, we see that there are
$\eta\in\UUUUU^{(1)}$ (resp.\ $\tUUUUU^{(1)}$) and $\zeta\in \UUUUU^{(-1)}$
(resp.\ $\tUUUUU^{(-1)}$) with $\mu=\eta+\zeta$.

Set $a_i=\eta(v_{2i-1})$ for $1\leq i\leq \ell$ and $a=\max\{a_1, \ldots, a_\ell\}$.
Since
\begin{eqnarray*}
\mu^+(C)&=&\eta^+(C)+\zeta^+(C),\\
\mu(-\infty)&=&\eta(-\infty)+\zeta(-\infty),\\
\mu^+(C)&=&n\ell=\ell\mu(-\infty),\\
\eta^+(C)+1&\leq&\ell\eta(-\infty)\\
\mbox{and}\hskip 1in\\
\zeta^+(C)-1&\leq&\ell\zeta(-\infty),
\end{eqnarray*}
we see that
$$
\eta^+(C)+1=\ell\eta(-\infty).
$$

Since $\eta(x)=1$ for any $x\in V\setminus\{v_1, v_3, \ldots, v_{2\ell-1}\}$ by
Lemma \ref{lem:triv}, we see that
$$
\eta^+(C)=\sum_{i=1}^\ell a_i+\ell+1.
$$
Thus, 
$$
\ell\eta(-\infty)=\sum_{i=1}^\ell a_i+\ell+2.
$$
Take $i$ with $a_i=a$ and a maximal clique in $G'$ (resp. maximal element of $\msKKK(G')$)
$K$ with $v_{2i-1}\in K$.
Then $\eta(x)=1$ for any $x\in K\setminus\{v_{2i-1}\}$ since $C$ does not have a chord.
Thus
$$
\eta^+(K)=a+m-1
$$
and we see that
$$
a+m=\eta^+(K)+1\leq\eta(-\infty).
$$
Therefore,
\begin{eqnarray*}
\ell\eta(-\infty)&=&\sum_{i=1}^\ell a_i+\ell+2\\
&\leq&\ell a+\ell+2\\
&=&\ell(a+m)-\ell(m-1)+2\\
&\leq&\ell\eta(-\infty)-\ell(m-1)+2.
\end{eqnarray*}
Thus,
$$
\ell(m-1)\leq 2.
$$
Since $m\geq 2$ and $\ell\geq 2$, we see that $\ell=m=2$, i.e., 
the length of $C$ is 5 and the size of maximal cliques in $G'$ is 2.
\end{proof}

In the rest of this paper, we use the following notation.
Let $G=(V,E)$ be a graph,
$G^{(1)}$, \ldots, $G^{(\ell)}$ be the connected components of $G$ and set
$J_d\define\{j\mid \omega(G^{(j)})=d\}$,
$J'_d\define\{j\mid$ the maximum size of elements in $\msKKK(G^{(j)})$ is $d\}$,
 for any positive integer $d$,
$I\define\{d\mid J_d\neq\emptyset\}=\{d_1,\ldots, d_u\}$, $d_1<\cdots<d_u$ and
$I'\define\{d\mid J'_d\neq\emptyset\}=\{d'_1,\ldots, d'_{u'}\}$, $d'_1<\cdots<d'_{u'}$.
Further, for any $d\in I$ (resp.\ $d\in I'$), we set
$V_d\define\bigcup_{j\in J_d}V(G^{(j)})$
(resp.\ $V'_d\define\bigcup_{j\in J'_d}V(G^{(j)})$)
and $G_d$ (resp.\ $G'_d$) the induced subgraph of $G$ by $V_d$ (resp. $V'_d$).
Note that $V'_1=V_1$, $V'_2=V_2$ and $V'_3=\bigcup_{d\geq 3}V_d$.

By Fact \ref{fact:gor cri}, Lemmas \ref{lem:gor pspec 1} and \ref{lem:gor pspec 2},
we see the following.

\begin{prop}
\mylabel{prop:gor pspec}
If $\ekhsg$ (resp.\ $\ekqsg$, $\ektsg$) is \gor\ on the punctured spectrum,
then 
$E_\KKK[\hstab(G_d)]$ (resp.\
${E_\KKK[\qstab(G_d)]}$,
${E_\KKK[\tstab(G'_d)]}$)
is \gor\ for any $d\in I$ (resp. $d\in I$, $d\in I'$).
\end{prop}

Now we consider h-perfect, perfect and t-perfect graphs.
If $G$ is perfect, then for every $d\in I$,
$E_\KKK[\qstab(G_d)]=E_\KKK[\stab(G_d)]$ is standard graded by 
\cite[Theorem 1.1]{oh}.
Further, if $G$ is h-perfect (resp. t-perfect) and $\eksg$ is \gor\ on the punctured spectrum,
then 
$E_\KKK[\stab(G_d)]$
is standard graded if $d\neq 2$ by \cite[Theorem 1.1]{oh}, 
Lemmas \ref{lem:gor pspec 1} and \ref{lem:gor pspec 2}.
Note that if $G$ is t-perfect, then there is no clique with size 4.
Thus $I=I'\subset\{1,2,3\}$ and $G'_d=G_d$ for any $d\in I$.

Now we state the following.

\begin{thm}
\mylabel{thm:gor pspec}
Let $G$ be an h-perfect graph.
Then $\eksg$ is \gor\ on the punctured spectrum if and only if
$E_\KKK[\stab(G_d)]$ is \gor\ for any $d\in I$
\end{thm}
\begin{proof}
The ``if'' part follows from Theorem \ref{thm:gor pspec general}
and the above mentioned fact.
Further the ``only if'' part follows from Proposition \ref{prop:gor pspec}.
\end{proof}
Since perfect or t-perfect graphs are h-perfect, we see the following.

\begin{cor}
\mylabel{cor:gor pspec perfect}
Let $G$ be a perfect or t-perfect graph.
Then  $\eksg$ is \gor\ on the punctured spectrum if and  only if
$E_\KKK[\stab(G_d)]$ is \gor\ for any $d\in I$.
\end{cor}

Now we state a criterion of nearly \gor\ property of h-perfect (resp.\ perfect, t-perfect)
graphs.
Note that if $G$ is h-perfect (resp.\ perfect, t-perfect) and $\eksg$ is \gor, then
$u=1$ by Fact \ref{fact:gor cri}.

\begin{thm}
\mylabel{thm:nearly gor}
Let $G$ be an h-perfect graph.
Set $R=\eksg$.
Then $R$ is not \gor\ and nearly \gor\ if and only if $u=2$, $d_2-d_1=1$ and
$E_\KKK[\stab(G_{d_i})]$ is \gor\ for $i=1$ and $2$.
\end{thm}
\begin{proof}
The ``only if'' part follows from Fact \ref{fact:gor cri}, Lemma \ref{lem:non pure non gor}
and Theorem \ref{thm:gor pspec}.

Next we prove the ``if'' part.
By Lemma \ref{lem:non pure non gor}, we see that $\trace(\omega_R)\subset R_{\geq 1}=\mmmm_R$.
We prove the reverse inclusion.
First consider the case where $d_1\geq3$.
This case follow from \cite[Theorem 4.15]{hhs}, since 
$a(E_\KKK[\stab(G_d)])=-d-1$ for any $d\in I$.

Next consider the case where $d_1=2$.
By assumption $d_2=3$ and $V$ is a disjoint union of $V_2$ and $V_3$.
Let $T^\mu$ be an arbitrary monomial in $\mmmm_R$.
Then by Fact \ref{fact:symb power}, we see that $\mu\in\UUUUU^{(0)}$ and $\mu(-\infty)\geq 1$.
Since $E_\KKK[\stab(G_3)]$ is standard graded, we see that there are 
$\mu_1$, $\mu_2\in\UUUUU^{(0)}(G_3)$ such that $\mu_1(-\infty)=1$ and $\mu_1(x)+\mu_2(x)=\mu(x)$
for any $x\in V_3^-$.
Define $\eta$, $\zeta\in\ZZZ^{V^-}$ by
$$
\eta(x)=
\left\{
\begin{array}{ll}
\mu(x)+1&\qquad\mbox{$x\in V_2$},\\
\mu_2(x)+1&\qquad\mbox{$x\in V_3$},\\
\mu(-\infty)+3&\qquad\mbox{$x=-\infty$},
\end{array}
\right.
$$
$$
\zeta(x)=
\left\{
\begin{array}{ll}
-1&\qquad\mbox{$x\in V_2$},\\
\mu_1(x)-1&\qquad\mbox{$x\in V_3$},\\
-3&\qquad\mbox{$x=-\infty$}.
\end{array}
\right.
$$
We show that $\eta\in\UUUUU^{(1)}$ and $\zeta\in\UUUUU^{(-1)}$.
It is clear that $\eta(x)\geq 1$ and $\zeta(x)\geq -1$ for any $x\in V$,
since $\mu\in\UUUUU^{(0)}$ and $\mu_1$, $\mu_2\in\UUUUU^{(0)}(G_3)$.
Next let $K$ be an arbitrary maximal clique in $G$.
If $K\subset V_2$, then $\#K=2$ and $\eta^+(K)=\mu^+(K)+2$.
Therefore, $\eta^+(K)+1=\mu^+(K)+3\leq\mu(-\infty)+3=\eta(-\infty)$.
Further, since $\zeta^+(K)=-2$, $\zeta^+(K)-1=-3=\zeta(-\infty)$.
If $K\subset V_3$, then $\#K=3$ and $\eta^+(K)=\mu_2^+(K)+3$.
Since $\mu_2(-\infty)=\mu(-\infty)-\mu_1(-\infty)=\mu(-\infty)-1$,
we see that
$\eta^+(K)+1=\mu_2^+(K)+4\leq \mu_2(-\infty)+4=\mu(-\infty)+3=\eta(-\infty)$.
Further, since $\mu_1(-\infty)=1$, we see that
$\zeta^+(K)=\mu_1^+(K)-3\leq\mu_1(-\infty)-3=-2$.
Therefore, $\zeta^+(K)-1\leq -3=\zeta(-\infty)$.
Next let $C$ be an odd cycle in $G$ without chord and length at least 5.
Then $C\subset V_2$ and the length of $C$ is 5 by Fact \ref{fact:gor cri}.
Therefore, since $\mu\in\UUUUU^{(0)}$, 
$\eta^+(C)=\mu^+(C)+5\leq 2\mu(-\infty)+5$.
Thus, we see that
$\eta^+(C)+1\leq 2\mu(-\infty)+6=2\eta(-\infty)$.
Further,
$\zeta^+(C)-1=-5-1=2\zeta(-\infty)$.

Therefore, we see that $\eta\in\UUUUU^{(1)}$ and $\zeta\in\UUUUU^{(-1)}$.
Since it is easily verified that $\eta+\zeta=\mu$, we see that 
$T^\mu\in\omega_R\omega_R^{-1}=\trace(\omega_R)$.
Since $T^\mu$ is an arbitrary monomial in $\mmmm_R$, we see that $\mmmm_R\subset\trace(\omega_R)$.

Finally we consider the case where $d_1=1$.
By assumption $d_2=2$ and $V$ is a disjoint union of $V_1$ and $V_2$.
Let $T^\mu$ be an arbitrary monomial in $\mmmm_R$.
Since $E_\KKK[\stab(G_{1})]$ is standard graded, we see that there are 
$\mu_1$, $\mu_2\in\UUUUU^{(0)}(G_1)$ such that $\mu_1(-\infty)=1$ and
$\mu_1(x)+\mu_2(x)=\mu(x)$ for any $x\in V_1^-$.
Note that $\mu_2(-\infty)=\mu(-\infty)-\mu_1(-\infty)=\mu(-\infty)-1$.
Set $\eta$, $\zeta\in\ZZZ^{V^-}$ by
$$
\eta(x)=
\left\{
\begin{array}{ll}
\mu_1(x)+1&\qquad\mbox{$x\in V_1$},\\
1&\qquad\mbox{$x\in V_2$},\\
3&\qquad\mbox{$x=-\infty$},
\end{array}
\right.
$$
$$
\zeta(x)=
\left\{
\begin{array}{ll}
\mu_2(x)-1&\qquad\mbox{$x\in V_1$},\\
\mu(x)-1&\qquad\mbox{$x\in V_2$},\\
\mu(-\infty)-3&\qquad\mbox{$x=-\infty$}.
\end{array}
\right.
$$
We show that $\eta\in \UUUUU^{(1)}$ and $\zeta\in \UUUUU^{(-1)}$.
It is clear that $\eta(x)\geq 1$ and $\zeta(x)\geq -1$ for any $x\in V$, since
$\mu(x)\geq 0$ for any $x\in V$ and $\mu_1(x)\geq0$ and $\mu_2(x)\geq0$ for any 
$x\in V_1$.
Next let $K$ be an arbitrary maximal clique in $G$.
If $K\subset V_1$, then $\#K=1$ and therefore $\eta^+(K)=\mu_1^+(K)+1$ and
$\zeta^+(K)=\mu_2^+(K)-1$.
Thus,
$\eta^+(K)+1=\mu_1^+(K)+2\leq\mu_1(-\infty)+2=3=\eta(-\infty)$
and
$\zeta^+(K)-1=\mu_2^+(K)-2\leq\mu_2(-\infty)-2=\mu(-\infty)-3=\zeta(-\infty)$.
If $K\subset V_2$, then $\#K=2$.
Thus,
$\eta^+(K)=2$ and $\zeta^+(K)=\mu^+(K)-2$.
Therefore,
$\eta^+(K)+1=3=\eta(-\infty)$
and 
$\zeta^+(K)-1=\mu^+(K)-3\leq\mu(-\infty)-3=\zeta(-\infty)$.
Finally, let $C$ be an odd cycle without chord and length at least 5.
Then $C\subset V_2$ and the length of $C$ is 5 by Fact \ref{fact:gor cri}.
Therefore,
$\eta^+(C)+1=5+1=6=2\eta(-\infty)$
and 
$\zeta^+(C)-1=\mu^+(C)-5-1=\mu^+(C)-6\leq 2\mu(-\infty)-6=2\zeta(-\infty)$.
Thus, we see that $\eta\in \UUUUU^{(1)}$ and $\zeta\in\UUUUU^{(-1)}$.
Since it is easily verified that $\eta+\zeta=\mu$, we see that 
$T^\mu\in\omega_R\omega_R^{-1}=\trace(\omega_R)$.
Since $T^\mu$ is an arbitrary monomial in $\mmmm_R$, we see that
$\mmmm_R\subset\trace(\omega_R)$.
\end{proof}
Since prefect and t-perfect graphs are h-perfect, we see the following.

\begin{cor}
\mylabel{cor:nearly gor perfect}
Let $G$ be a perfect or t-perfect graph.
Then $\eksg$ is not \gor\ and nearly \gor\ if and only if
$u=2$, $d_2-d_1=1$ and $E_\KKK[\stab(G_{d_i})]$ is \gor\ for $i=1$, $2$.
\end{cor}



%

\end{document}